\documentclass{amsart}[11pt]
\usepackage{amssymb}
\usepackage{verbatim}
\usepackage{hyperref}

\usepackage[margin=1.35in]{geometry}

\newtheorem{exercise}{Exercise}
\newtheorem{lem}[exercise]{Lemma}
\newtheorem{prop}[exercise]{Proposition}

\newtheorem{theorem}[exercise]{Theorem}

\theoremstyle{definition}
\newtheorem{example}{Example}

\newcommand{\C}{\mathbb{C}}
\newcommand{\code}{\mathrm{code}}

\newcommand{\Fl}{\mathrm{Fl}}
\newcommand{\Id}{\mathrm{Id}}
\newcommand{\ip}[2]{\langle #1\,,\,#2\rangle}
\newcommand{\la}{\lambda}
\newcommand{\ladot}{{\lambda^\bullet}}

\newcommand{\nd}{{n_\bullet}}
\newcommand{\pt}{\mathrm{pt}}

\newcommand{\qSchub}{\tilde{\Schub}}
\newcommand{\qt}{\theta}

\newcommand{\Schub}{\mathfrak{S}}

\newcommand{\wleq}{\le^L}

\newcommand{\Y}{\mathbb{Y}}
\newcommand{\Z}{\mathbb{Z}}

\title[Quantum double Schubert polynomials]{Quantum double Schubert polynomials represent Schubert classes}
\author{Thomas Lam}
\address{Department of Mathematics,
University of Michigan, 530 Church St., Ann Arbor, MI 48109 USA}
\email{tfylam@umich.edu}
 \thanks{T.L. was supported by NSF grant DMS-0901111, and by a Sloan Fellowship.}
\author{Mark Shimozono}
\address{Department of Mathematics, Virginia Tech, Blacksburg, VA 24061-0123 USA}
\email{mshimo@vt.edu}
\thanks{M.S. was supported by NSF DMS-0652641 and DMS-0652648.}
\begin{document}
\begin{abstract} The quantum double Schubert polynomials studied by Kirillov and Maeno, and by Ciocan-Fontanine and Fulton, are shown to
represent Schubert classes in Kim's presentation of the equivariant quantum cohomology of the flag variety.  We define parabolic analogues of quantum double Schubert polynomials, and show that they represent Schubert classes in the equivariant quantum cohomology of partial flag varieties.  For complete flags Anderson and Chen \cite{AC} have announced a proof with different methods.
\end{abstract}

\maketitle


\section{Introduction}
Let $H^*(\Fl)$, $H^T(\Fl)$, $QH^*(\Fl)$, and $QH^T(\Fl)$ be the 
ordinary, $T$-equivariant, quantum, and $T$-equivariant quantum cohomology rings of the variety $\Fl=\Fl_n$
of complete flags in $\C^n$ where $T$ is the maximal torus of $GL_n$.  All cohomologies are with $\Z$ coefficients.   The flag variety $\Fl_n$ has a stratification by Schubert varieties $X_w$, labeled by permutations $w \in S_n$, which gives rise to {\it Schubert bases} for each of these rings.

This paper is concerned with the problem of finding polynomial representatives for the Schubert bases in a ring presentation of these (quantum) cohomology rings.  These ring presentations are due to Borel \cite{Bo} in the classical case, and Ciocan-Fontanine \cite{Cio}, Givental and Kim \cite{GK} and Kim \cite{Kim} in the quantum case.  This is a basic problem in classical and quantum Schubert calculus.

This problem has been solved in the first three cases: the Schubert polynomials are known to represent Schubert classes in $H^*(\Fl_n)$ by work of Bernstein, Gelfand, and Gelfand \cite{BGG} and Lascoux and Sch\"{u}tzenberger \cite{LS}; the double Schubert polynomials, also due to Lascoux and Sch\"{u}tzenberger, represent Schubert classes in $H^T(\Fl_n)$ (see for example \cite{Bi}); and the quantum Schubert polynomials of Fomin, Gelfand, and Postnikov \cite{FGP} represent Schubert classes in $QH^*(\Fl)$.  These polynomials are the subject of much research by combinatorialists and geometers and we refer the reader to these references for a complete discussion of these ideas.  Our first main result (Theorem \ref{T:main}) is that the quantum double Schubert polynomials of \cite{KM,CF} represent equivariant quantum Schubert classes in $QH^T(\Fl)$. Anderson and Chen \cite{AC} have announced a proof of this result
using the geometry of Quot schemes.

Now let $SL_n/P$ be a partial flag variety, where $P$ denotes a parabolic subgroup of $SL_n$.  In non-quantum Schubert calculus, the functorality of (equivariant) cohomology implies that the (double) Schubert polynomials labeled by minimal length coset representatives again represent Schubert classes in $H^*(SL_n/P)$ or $H^T(SL_n/P)$.  This is not the case in quantum cohomology.  Ciocan-Fontanine \cite{Cio2} solved the corresponding problem in $QH^*(SL_n/P)$, extending Fomin, Gelfand, and Postnikov's work to the parabolic case.  Here we introduce the parabolic quantum double Schubert polynomials.  We show that these polynomials represent Schubert classes in the torus-equivariant quantum cohomology $QH^T(SL_n/P)$ of a partial flag variety.  
Earlier, Mihalcea \cite{Mi2} had found polynomial representatives for the Schubert basis in the special case of the equivariant quantum cohomology of the Grassmannian.

\medskip
{\bf Acknowledgements.} We thank Linda Chen for communicating to us her joint work with Anderson \cite{AC}.

\section{The (quantum) cohomology rings of flag manifolds}

\subsection{Presentations}
Let $x=(x_1,\dotsc,x_n)$, $a=(a_1,\dotsc,a_n)$, and $q=(q_1,\dotsc,q_{n-1})$
be indeterminates. We work in the graded polynomial ring $\Z[x;q;a]$ with
$\deg(x_i)=\deg(a_i)=1$ and $\deg(q_i)=2$.
Let $S=\Z[a]$ be identified with $H^T(\pt)$ and let $e_j(a_1,\dotsc,a_n)$ be the elementary symmetric polynomial.
Let $C_n$ be the tridiagonal $n\times n$ matrix with entries $x_i$ on the diagonal, $-1$ on the superdiagonal,
and $q_i$ on the subdiagonal. Define the polynomials $E_j^n\in \Z[x;q]$ by
\begin{align*}
  \det(C_n-t \,\Id) = \sum_{j=0}^n (-t)^{n-j} E_j^n.
\end{align*}
Let $J$ (resp. $J^a$, $J^q$, $J^{qa}$) be the ideal in $\Z[x]$ (resp. $S[x]$, $\Z[x;q]$, $S[x;q]$)
generated by the elements $e_j(x)$ (resp. $e_j(x)-e_j(a)$; $E_j^n$; $E_j^n-e_j(a)$) for $1\le j\le n$;
in all cases the $j$-th generator is homogeneous of degree $j$.  We have
\begin{align}
\label{E:H}
  H^*(\Fl)&\cong \Z[x]/J \\
\label{E:HT}
  H^T(\Fl)&\cong S[x]/J^a \\
\label{E:QH}
  QH^*(\Fl)&\cong \Z[x;q]/J^q \\
\label{E:QHT}
  QH^T(\Fl)&\cong S[x;q]/J^{qa}.
\end{align}
as algebras over $\Z$, $S$, $\Z[q]$, and $S[q]$ respectively.  The presentation of $H^*(\Fl)$ is a classical result due to Borel.
The presentations of $QH^*(\Fl)$ and $QH^T(\Fl)$ are due to Ciocan-Fontanine \cite{Cio},
Givental and Kim \cite{GK} and Kim \cite{Kim}.

\subsection{Schubert bases}
Let $X_w=\overline{B_-wB/B}\subset\Fl$ be the opposite Schubert variety, where $w\in W = S_n$ is a permutation, $B\subset SL_n$
is the upper triangular Borel and $B_-$ the opposite Borel. The ring $H^*(\Fl)$ (resp. $H^T(\Fl)$)
has a basis over $\Z$ (resp. $S$) denoted $[X_w]$ (resp. $[X_w]_T$) associated with the Schubert varieties.

Given three elements $u,v,w\in W$ and an element of the coroot lattice $\beta\in Q^\vee$
one may define a genus zero Gromov-Witten invariant $c^w_{uv}(\beta) \in\Z_{\ge0}$ (see \cite{GK,Kim})
and an associative ring $QH^*(\Fl)$ with $\Z[q]$-basis $\{\sigma^w\mid w\in W\}$
(called the quantum Schubert basis) such that 
\begin{align*}
  \sigma^u \sigma^v = \sum_{w,\beta} q_\beta c^w_{uv}(\beta) \sigma^w
\end{align*}
where $q_\beta = \prod_{i=1}^{n-1} q_i^{k_i}$ where $\beta=\sum_{i=1}^{n-1} k_i \alpha_i^\vee$ for $k_i\in\Z$.
Similarly there is a basis of a ring $QH^T(\Fl)$ with $S[q]$-basis given by
the equivariant quantum Schubert classes $\sigma^w_T$, defined using equivariant Gromov-Witten invariants,
which are elements of $S$.

We shall use the following characterization of $QH^T(\Fl)$ and its Schubert basis $\{\sigma^w_T\mid w\in W\}$ due to Mihalcea \cite{Mi}.
Let $\Phi^+$ be the set of positive roots and $\rho$ the half sum of positive roots.
For $w\in W$ define
\begin{align*}
A_w &= \{\alpha\in \Phi^+\mid ws_\alpha\gtrdot w \} \\
B_w &= \{\alpha\in \Phi^+\mid \ell(ws_\alpha)=\ell(w)+1-\ip{\alpha^\vee}{2\rho} \}.
\end{align*}
Let $\omega_i(a)=a_1+\dotsm+a_i\in S$ be the fundamental weight.
We write $A_w^n$ and $B_w^n$ to emphasize that the computation pertains to $\Fl=SL_n/B$.

\begin{theorem} \label{T:QHTchar} \cite[Corollary 8.2]{Mi}
For $w\in S_n$ and $1\le i\le n-1$ a Dynkin node, the equivariant quantum Schubert classes $\sigma^w_T$
satisfy the equivariant quantum Chevalley-Monk formula
\begin{align}\label{E:EQChev}
  \sigma^{s_i}_T \sigma^w_T &= (-\omega_i(a) + w\cdot\omega_i(a)) + 
  \sum_{\alpha\in A_w^n} \ip{\alpha^\vee}{\omega_i} \sigma_T^{ws_\alpha} +
  \sum_{\alpha\in B_w^n} q_{\alpha^\vee} \ip{\alpha^\vee}{\omega_i} \sigma_T^{ws_\alpha}.
\end{align}
Moreover these structure constants determine the Schubert basis $\{\sigma^w_T\mid w\in S_n\}$
and the ring $QH^T(SL_n/B)$ up to isomorphism as $\Z[q_1,\dotsc,q_{n-1};a_1,\dotsc,a_n]$-algebras.
\end{theorem}

\section{Quantum double Schubert polynomials}
Now we work with infinite sets of variables $x=(x_1,x_2,\dotsc)$,
$q=(q_1,q_2,\dotsc)$, and $a=(a_1,a_2,\dotsc)$.

\subsection{Various Schubert polynomials}
Let $\partial_i^a = \alpha_i^{-1}(1-s_i^a)$ be the divided difference operator, where $\alpha_i=a_i-a_{i+1}$
and $s_i^a$ is the operator that exchanges $a_i$ and $a_{i+1}$. Since the operators $\partial_i=\partial_i^a$ satisfy the
braid relations one may define $\partial_w=\partial_{i_1}\dotsm \partial_{i_\ell}$ where $w=s_{i_1}\dotsm s_{i_\ell}$
is a reduced decomposition. For $w\in S_n$ define the double Schubert polynomial $\Schub_w(x;a)$ \cite{LS}
and the quantum double Schubert polynomial $\qSchub_w(x;a)\in S[x;q]$ \cite{KM,CF} by
\begin{align}
\label{E:dSchubdef}
  \Schub_w(x;a) &= (-1)^{\ell(ww_0^{(n)})}\partial_{ww_0^{(n)}}^a \prod_{i=1}^{n-1} \prod_{j=1}^i (x_j-a_{n-i}) \\
\label{E:qdSchubdef}
  \qSchub_w(x;a) &= (-1)^{\ell(ww_0^{(n)})}\partial_{ww_0^{(n)}}^a \prod_{i=1}^{n-1} \det(C_i-a_{n-i}\Id)
\end{align}
where $w_0^{(n)}\in S_n$ is the longest element.\footnote{This is not the standard definition of double Schubert polynomial.
However it is easily seen to be equivalent using, say, the identity $\Schub_{w^{-1}}(x;a)=\Schub_w(-a;-x)$ \cite{Mac}.}
Note that it is equivalent to define $\Schub_w(x;a)$ by setting the $q_i$ variables to zero in $\qSchub_w(x;a)$.

Let $S_\infty = \bigcup_{n\ge1} S_n$ be the infinite symmetric group under the embeddings $i_n:S_n\to S_{n+1}$
that add a fixed point at the end of a permutation. Due to the stability property \cite{KM}
$\qSchub_{i_n(w)}(x;a) = \qSchub_w(x;a)$ for $w\in S_n$, the quantum double Schubert polynomials
$\qSchub_w(x;a)$ are well-defined for $w\in S_\infty$. Similarly, $\Schub_w(x;a)$ is well-defined for $w\in S_\infty$.

For $w\in S_\infty$, define the (resp. quantum) Schubert polynomial $\Schub_w(x) = \Schub_w(x;0)$ (resp. $\qSchub_w(x) = \qSchub_w(x;0)$)
by setting the $a_i$ variables to zero in the (resp. quantum) double Schubert polynomial.
Note that $\Schub_w(x)$, $\Schub_w(x;a)$, $\qSchub_w(x)$, and $\qSchub_w(x;a)$ are all homogeneous of degree $\ell(w)$.
The original definition of quantum Schubert polynomial in \cite{FGP} is different.
However their definition and the one used here, are easily seen to be equivalent \cite{KM}, due to the
commutation of the divided differences in the $a$ variables and the quantization map $\qt$ of \cite{FGP},
which we review in \S \ref{SS:qt}.

\begin{lem} \label{L:lead} \cite{Mac} For $w\in S_\infty$,
the term that is of highest degree in the $x$ variables and then
is the reverse lex leading such term, in any of $\Schub_w(x)$,
$\Schub_w(x;a)$, $\qSchub_w(x)$, and $\qSchub_w(x;a)$, is the monomial $x^\code(w)$,
where 
\begin{align*}
  \code(w) &= (c_1,c_2,\dotsc) \\
  c_i &= |\{j\in\Z_{>0} \mid \text{$i<j$ and $w(j)<w(i)$} \}| \qquad\text{for $i\in\Z_{>0}$.}
\end{align*}
\end{lem}

\begin{lem} \label{L:code} \cite{Mac} There is a bijection from $S_\infty$ to the
set of tuples $(c_1,c_2,\dotsc)$ of nonnegative integers, almost all zero,
given by $w\mapsto\code(w)$. Moreover it restricts to a bijection from $S_n$
to the set of tuples $(c_1,\dotsc,c_n)$ such that $0\le c_i \le n-i$ for all $0\le i\le n$.
\end{lem}

\begin{lem} \label{L:basisquot} \
\begin{enumerate}
\item
$\{\Schub_w(x)\mid w\in S_n\}$ is a $\Z$-basis of $\Z[x_1,\dotsc,x_n]/J_n$.
\item
$\{\Schub_w(x;a)\mid w\in S_n\}$ is a $\Z[a]$-basis of $\Z[x_1,\dotsc,x_n;a_1,\dotsc,a_n]/J_n^a$.
\item
$\{\qSchub_w(x)\mid w\in S_n\}$ is a $\Z[q]$-basis of $\Z[x_1,\dotsc,x_n;q_1,\dotsc,q_{n-1}]/J_n^q$.
\item
$\{\qSchub_w(x;a)\mid w\in S_n\}$ is a $\Z[q,a]$-basis of $\Z[x_1,\dotsc,x_n;q_1,\dotsc,q_{n-1};a_1,\dotsc,a_n]/J_n^{q,a}$.
\end{enumerate}
\end{lem}
\begin{proof} Since in each case the highest degree part of the $j$-th ideal generator in the $x$
variables is $e_j(x_1,\dotsc,x_n)$, any polynomial may be reduced modulo the ideal until its leading term in the
$x$ variables is $x^\gamma$ where $\gamma=(\gamma_1,\dotsc,\gamma_n)\in\Z_{\ge0}^n$
with $\gamma_i \le n-i$ for $1\le i\le n$. But these are the leading terms of the
various kinds of Schubert polynomials.
\end{proof}

\subsection{Geometric bases}
Under the isomorphism \eqref{E:H} (resp. \eqref{E:HT}, \eqref{E:QH})
the Schubert basis $[X_w]$ (resp. $[X_w]_T$, $\sigma^w$)
corresponds to $\Schub_w(x)$ (resp. $\Schub_w(x;a)$, $\qSchub_w(x)$),
by \cite{LS,BGG} for $H^*(\Fl)$, \cite{Bi} for $H^T(\Fl)$, and \cite{FGP} for $QH^*(\Fl)$.  Our first main result is:

\begin{theorem} \label{T:main}
Under the $S[q]$-algebra isomorphism \eqref{E:QHT} the quantum equivariant Schubert basis element
$\sigma^w_T$ corresponds to the quantum double Schubert polynomial $\qSchub_w(x;a)$.
\end{theorem}

Theorem \ref{T:main} is proved in Section \ref{S:proofmain}.

\subsection{Stable quantization}
\label{SS:qt}

This section follows \cite{FGP}. 
Let $e_i^r = e_i(x_1,x_2,\dotsc,x_r)\in\Z[x]$ be the elementary symmetric polynomial for integers $0\le i\le r$.
By \cite[Prop. 3.3]{FGP}, $\Z[x]$ has a $\Z$-basis of \textit{standard monomials} $e_I=\prod_{r\ge1} e_{i_r}^r$ 
where $I=(i_1,i_2,\dotsc)$ is a sequence of nonnegative integers, almost all zero,
with $0\le i_r \le r$ for all $r\ge 1$. 

The stable quantization map is the $\Z[q]$-module automorphism $\qt$ of $\Z[x;q]$ given by
\begin{align*}
  e_I \mapsto E_I := \prod_{j\ge1} E_{i_j}^j.
\end{align*}
By \cite{FGP,CF,KM} we have (this is the definition of quantum Schubert polynomial in \cite{FGP})
\begin{align}\label{E:qtSchub}
  \qt(\Schub_w(x)) = \qSchub_w(x)\qquad\text{for all $w\in S_\infty$.}
\end{align}
The map $\qt$ is extended by $\Z[a]$-linearity to a $\Z[q,a]$-module automorphism
of $\Z[x,q,a]$.

\subsection{Cauchy formulae} 
The double Schubert polynomials satisfy \cite{Mac}
\begin{align}\label{E:dsCauchy}
  \Schub_w(x;a) = \sum_{v\wleq w} \Schub_{vw^{-1}}(-a) \Schub_v(x)
\end{align}
where $v\wleq w$ denotes the left weak order, defined by $\ell(wv^{-1})+\ell(v)=\ell(w)$.
For a geometric explanation of this identity see \cite{A}. We have \cite{KM,CF}
\begin{align}\label{E:quantCauchy}
  \qt(\Schub_w(x;a)) &= \qSchub_w(x;a) = \sum_{v\wleq w} \Schub_{vw^{-1}}(-a) \qSchub_v(x).
\end{align}
The first equality follows from the divided difference definitions of $\qSchub_w(x;a)$
and $\Schub_w(x;a)$ and the commutation of the divided differences in the $a$-variables with quantization.
The second equality follows from quantizing \eqref{E:dsCauchy}.

We require the explicit formulae for Schubert polynomials indexed by simple reflections.

\begin{lem}\label{L:qdsreflection} We have
\begin{align}
\label{E:Schubreflect}
  \Schub_{s_i}(x) &= \omega_i(x) = x_1+x_2+\dotsm+x_i \\
\label{E:qSchubreflect}
  \qSchub_{s_i}(x) &= \Schub_{s_i}(x) \\
\label{E:qdSchubreflect}
  \qSchub_{s_i}(x;a) &= \omega_i(x) - \omega_i(a).
\end{align}
\end{lem}
\begin{proof} Since $s_i$ is an $i$-Grassmannian permutation with associated
partition consisting of a single box, its Schubert polynomial
is the Schur polynomial \cite{Mac}
$\Schub_{s_i}(x) = S_1[x_1,\dotsc,x_i] = \omega_i(x)$, proving \eqref{E:Schubreflect}.
For \eqref{E:qSchubreflect} we have $\qSchub_{s_i}(x) = \qt(\Schub_{s_i}(x))=\qt(e_{1,i})=E_{1,i} = \omega_i(x)$. 
For \eqref{E:qdSchubreflect}, by \eqref{E:quantCauchy} we have 
$\qSchub_{s_i}(x;a) = \qSchub_{s_i}(x) - \Schub_{s_i}(a) = \omega_i(x) - \omega_i(a)$
as required.
\end{proof}

\section{Chevalley-Monk rules for Schubert polynomials}
The Chevalley-Monk formula describes the product of a divisor class and an arbitrary Schubert class in the cohomology ring $H^*(\Fl)$.  The goal of this section is to establish the Chevalley-Monk rule for quantum double Schubert polynomials.  The Chevalley-Monk rules for (double, quantum, quantum double) Schubert polynomials should be viewed as product rules for
the cohomologies of an infinite-dimensional flag ind-variety $\Fl_\infty$ of type $A_\infty$ with Dynkin node set $\Z_{>0}$ and simple bonds between $i$ and $i+1$ for all $i\in \Z_{>0}$.

Let $\Phi^+= \{\alpha_{ij}=a_i-a_j\mid 1\le i<j\}$ be the set of positive roots.
Let $\alpha_{ij}^\vee=a_i-a_j$ and $\alpha_i^\vee=\alpha_{i,i+1}^\vee$ by abuse of notation.
Let $s_{ij}=s_{\alpha_{ij}}$ for $1\le i<j$. 
For $w\in S_\infty$ let $A_w$ and $B_w$ be defined as before Theorem \ref{T:QHTchar} but using the infinite set of positive roots $\Phi^+$
and letting $\rho=(0,-1,-2,\dotsc)$. To distinguish between the finite and limiting infinite cases,
we denote by $A_w^n$ the set $A_w$ for $w\in S_n$ which uses the positive roots of $SL_n$.

\subsection{Schubert polynomials}

\begin{prop} \label{P:Chev} \cite{Che,Mo}.
For $w\in S_n$ and $1\le i\le n-1$, in $H^*(\Fl)$ we have
\begin{align*}
  [X_{s_i}] [X_w] = \sum_{\alpha\in A_w^n} \ip{\alpha^\vee}{\omega_i} [X_{ws_\alpha}].
\end{align*}
\end{prop}

\begin{prop}\label{P:poly} \cite{Mac} 
For $w\in S_\infty$ and $i\in \Z_{>0}$ the Schubert polynomials satisfy the identity in $\Z[x]$ given by
\begin{align}\label{E:poly}
  \Schub_{s_i}(x) \Schub_w(x) &= \sum_{\alpha\in A_w} \ip{\alpha^\vee}{\omega_i} \Schub_{ws_\alpha}(x).
\end{align}
\end{prop}

\begin{example} It is necessary to take a large-rank limit $(n \gg 0)$ to compare Propositions \ref{P:Chev} and \ref{P:poly}.
Let $n=2$. We have $[X_{s_1}]^2 = 0$ in $H^*(\Fl_2)$ since $A_{s_1}=\emptyset$ for $SL_2$.
Lifting to polynomials we have $\Schub_{s_1}^2 = x_1^2 = \Schub_{s_2s_1}$
since $A_{s_1}=\{\alpha_{13}\}$, which is not a positive root for $SL_2$.
Note that $\Schub_{s_2s_1}\in J_2$ and $s_2s_1\in S_3\setminus S_2$. In $H^*(\Fl_n)$ for $n\ge3$ we have
$[X_{s_1}]^2 = [X_{s_2s_1}]$. 
\end{example}

\subsection{Quantum Schubert polynomials}

\begin{prop} \label{P:qChev} \cite{FGP} For $w\in S_n$ and $1\le i\le n-1$, in $QH^*(\Fl_n)$ we have
\begin{align*}
  \sigma^{s_i}\sigma^w = \sum_{\alpha\in A_w^n} \ip{\alpha^\vee}{\omega_i} \sigma^{ws_\alpha} +
  \sum_{\alpha\in B_w^n} q_{\alpha^\vee} \ip{\alpha^\vee}{\omega_i} \sigma^{ws_\alpha}.
\end{align*}
\end{prop}

\begin{prop} \label{P:Qpoly} \cite{FGP}
For $i\in \Z_{>0}$ and $w\in S_\infty$ the quantum Schubert polynomials satisfy the identity in $\Z[x;q]$ given by
\begin{align}\label{E:Qpoly}
  \qSchub_{s_i}(x) \qSchub_w(x) &= \sum_{\alpha\in A_w} \ip{\alpha^\vee}{\omega_i} \qSchub_{ws_\alpha}(x) +
  \sum_{\alpha\in B_w} q_{\alpha^\vee} \ip{\alpha^\vee}{\omega_i} \qSchub_{ws_\alpha}(x).
\end{align}
\end{prop}

\subsection{Double Schubert polynomials}

\begin{prop}\label{P:EChev} \cite{KK} \cite{Rob}. For $w\in S_n$
and $1\le i\le n-1$, in $H^T(\Fl_n)$ we have
\begin{align}
  [X_{s_i}]_T [X_w]_T = (-\omega_i(a) + w\cdot \omega_i(a)) [X_w]_T +\sum_{\alpha\in A_w^n} \ip{\alpha^\vee}{\omega_i} [X_{ws_\alpha}]_T.
\end{align}
\end{prop}

The following is surely known but we include a proof for lack of a known reference.

\begin{prop}\label{P:ECpoly} For $w\in S_\infty$ and $i\in\Z_{>0}$, the double Schubert polynomials satisfy the identity in $\Z[x,a]$ given by
\begin{align}\label{E:ECpoly}
  \Schub_{s_i}(x;a) \Schub_w(x;a) = (-\omega_i(a)+w\cdot \omega_i(a)) \Schub_w(x;a)+
  \sum_{\alpha\in A_w} \ip{\alpha^\vee}{\omega_i} \Schub_{ws_\alpha}(x;a).
\end{align}
\end{prop}
\begin{proof} Fix $w \in S_\infty$.  We observe that the set $A_w$ is finite. Let $N$ be large enough so that
all appearing terms make sense for $S_N$. 
By \cite{Bi} under the isomorphism \eqref{E:HT}, $[X_w]_T\mapsto \Schub_w(x;a)+J_N^a$ for $w\in S_N$.
By Proposition \ref{P:EChev} for $H^T(\Fl_N)$, equation \eqref{E:ECpoly} holds modulo an element $f\in J_N^a$.
We may write $f=\sum_{v\in S_\infty} b_v \Schub_v(x;a)$ where $b_v\in \Z[a]$ and only finitely many
are nonzero. Choose $n\ge N$ large enough so that $v\in S_n$ and $b_v\in \Z[a_1,\dotsc,a_n]$
for all $v$ with $b_v\ne 0$. Applying Proposition \ref{P:EChev} again for $H^T(\Fl_n)$ 
we deduce that $f\in J_n^a$. By Lemma \ref{L:basisquot} it follows that $f=0$ as required.
\end{proof}

\subsection{Quantum double Schubert polynomials}

Theorem \ref{T:QHTchar} gives the equivariant quantum Chevalley-Monk rule for $QH^T(\Fl_n)$.
We cannot use the multiplication rule in Theorem \ref{T:QHTchar} directly
because we are trying to prove that the quantum double Schubert polynomials represent Schubert classes.  We deduce the following product formula by cancelling down to the equivariant case which was proven above.

\begin{prop} \label{P:QEChevalley} The quantum double Schubert polynomials satisfy the equivariant quantum Chevalley-Monk
rule in $\Z[x,q,a]$: for all $w\in S_\infty$ and $i\ge1$ we have
\begin{align}\label{E:EQCpoly}
  \qSchub_{s_i}(x;a) \qSchub_w(x;a) &= (-\omega_i(a)+w\cdot\omega_i(a)) \qSchub_w(x;a) +
  \sum_{\alpha\in A_w} \ip{\alpha^\vee}{\omega_i} \qSchub_{ws_\alpha}(x;a) \\
\notag  &+ \sum_{\alpha\in B_w}  q_{\alpha^\vee}  \ip{\alpha^\vee}{\omega_i}\qSchub_{ws_\alpha}(x;a).
\end{align}
\end{prop}
\begin{proof} Starting with \eqref{E:ECpoly} and using 
Lemma \ref{L:qdsreflection} and \eqref{E:dsCauchy} we have
\begin{align*}
  0 &= -\Schub_{s_i}(x;a)\Schub_w(x;a) + (-\omega_i(a)+w\cdot\omega_i(a))\Schub_w(x;a)+\sum_{\alpha\in A_w} \ip{\alpha^\vee}{\omega_i} \Schub_{ws_\alpha}(x;a) \\
  &=(\omega_i(a)-\Schub_{s_i}(x))\Schub_w(x;a) + (-\omega_i(a)+w\cdot\omega_i(a))\Schub_w(x;a)+\sum_{\alpha\in A_w} \ip{\alpha^\vee}{\omega_i} \Schub_{ws_\alpha}(x;a) \\
  &=-\Schub_{s_i}(x)\sum_{v\wleq w} \Schub_{vw^{-1}}(-a)\Schub_v(x) + (w\cdot\omega_i(a))\Schub_w(x;a)+\sum_{\alpha\in A_w} \ip{\alpha^\vee}{\omega_i} \Schub_{ws_\alpha}(x;a) \\
  &= -\sum_{v\wleq w} \Schub_{vw^{-1}}(-a)\sum_{\alpha\in A_v} \ip{\alpha^\vee}{\omega_i}\Schub_{vs_\alpha}(x) + (w\cdot\omega_i(a))\Schub_w(x;a)+\sum_{\alpha\in A_w} \ip{\alpha^\vee}{\omega_i} \Schub_{ws_\alpha}(x;a).
\end{align*}
Quantizing and rearranging, we have
\begin{align*}
&\,(w\cdot\omega_i(a))\qSchub_w(x;a)+\sum_{\alpha\in A_w} \ip{\alpha^\vee}{\omega_i} \qSchub_{ws_\alpha}(x;a) \\
&= \sum_{v\wleq w} \Schub_{vw^{-1}}(-a)\sum_{\alpha\in A_v} \ip{\alpha^\vee}{\omega_i}\qSchub_{vs_\alpha}(x) \\
&= \sum_{v\wleq w} \Schub_{vw^{-1}}(-a)\left(\qSchub_{s_i}(x) \qSchub_v(x)-\sum_{\alpha\in B_v} \ip{\alpha^\vee}{\omega_i}\qSchub_{vs_\alpha}(x)
\right) \\
&= \qSchub_{s_i}(x) \qSchub_w(x;a)-\sum_{v\wleq w} \Schub_{vw^{-1}}(-a)\sum_{\alpha\in B_v} \ip{\alpha^\vee}{\omega_i}\qSchub_{vs_\alpha}(x).
\end{align*}
Therefore to prove \eqref{E:EQCpoly} it suffices to show that
\begin{align}\label{E:corrections}
&\,\,\,\sum_{v\wleq w} \Schub_{vw^{-1}}(-a) \sum_{\alpha\in B_v} q_{\alpha^\vee} \ip{\alpha^\vee}{\omega_i} \qSchub_{vs_\alpha}(x) \\
\notag &=  \sum_{\alpha\in B_w}  \ip{\alpha^\vee}{\omega_i} q_{\alpha^\vee} 
\sum_{v\wleq ws_\alpha} \Schub_{vs_\alpha w^{-1}}(-a) \qSchub_v(x).
\end{align}
Let 
\begin{align*}
  A &= \{(v,\alpha)\in W\times \Phi^+ \mid \text{$v\wleq w$ and $\alpha\in B_v$} \} \\
  B &= \{(u,\alpha)\in W\times \Phi^+ \mid \text{$u\wleq ws_\alpha$ and $\alpha\in B_w$} \}.
\end{align*}
To prove \eqref{E:corrections} it suffices to show that there is a bijection $A\to B$ given by
$(v,\alpha)\mapsto (vs_\alpha,\alpha)$.

Let $(v,\alpha)\in A$. Then $w = (wv^{-1})(v)$ is length-additive since $v\wleq w$
and $v = (vs_\alpha)(s_\alpha)$ is length-additive because $\ell(s_\alpha)=\ip{\alpha^\vee}{2\rho}-1$
and $\alpha\in B_v$. Therefore $w = (wv^{-1})(vs_\alpha)(s_\alpha)$ is length-additive.
It follows that $ws_\alpha = (wv^{-1})(vs_\alpha)$ is length-additive
and that $vs_\alpha \wleq ws_\alpha$.
Moreover we have $\ell(ws_\alpha)=\ell(wv^{-1})+\ell(vs_\alpha)=\ell(wv^{-1})+\ell(v)+1-\ip{\alpha^\vee}{2\rho}
=\ell(w)+1-\ip{\alpha^\vee}{2\rho}$. Therefore $(vs_\alpha,\alpha)\in B$.

Conversely suppose $(u,\alpha)\in B$. Let $v=u s_\alpha$.
Arguing as before, $w = (ws_\alpha u^{-1})(u)(s_\alpha) = (w v^{-1})(v)$
are length-additive. We deduce that $v\wleq w$ and that 
$\ell(v)=\ell(w)-\ell(wv^{-1})=\ell(ws_\alpha)+1-\ip{\alpha^\vee}{2\rho}-\ell(wv^{-1})=\ell(vs_\alpha)+1-\ip{\alpha^\vee}{2\rho}$
so that $\alpha\in B_v$ as required.
\end{proof}

\section{Proof of Theorem \ref{T:main}}
\label{S:proofmain}

Let $I^a$ be the ideal in $\Z[x,a]$ generated by $e_i^p(x)-e_i^p(a)$ for $p\ge n$ and $i\ge1$,
and $a_i$ for $i > n$. Let $J^a\subset\Z[x,a]$ be the $\Z[a]$-submodule spanned by
$\Schub_w(x;a)$ for $w\in S_\infty\setminus S_n$ and $a_i \Schub_u(x;a)$ for $i >n$
and any $u\in S_\infty$. We shall show that $I^a=J^a$. Let $c_{i,p}=s_{p+1-i}\dotsm s_{p-2} s_{p-1} s_p\in S_\infty\setminus S_p$
be the cycle of length $i$. We note that the family
$\{e_i^p(x)-e_i^p(a)\mid 1\le i\le p\}$ is unitriangular over $\Z[a]$ with
the family $\{\Schub_{c_{i,p}}(x;a)\mid 1\le i\le p\}$. Since $\Z[x,a] = \bigoplus_{u\in S_\infty} \Z[a] \Schub_u(x;a)$,
to show that $I^a\subset J^a$ it suffices to show that $\Schub_u(x;a) \Schub_{c_{i,p}}(x,a)\in J^a$
for all $p\ge n$, $i\ge1$, and $u\in S_\infty$. But this follows from the fact that the product of
$\Schub_u(x,a)\Schub_v(x;a)$ is a $\Z[a]$-linear combination of $\Schub_w(x;a)$ where $w\ge u$ and $w\ge v$.

Let $K$ be the ideal in $\Z[x,a]$ generated by $a_i$ for $i > n$.
Then $I^a/K$ has $\Z[a_1,\dotsc,a_n]$-basis given by standard monomials $e_I$
with $i_r>0$ for some $r\ge n$, while $J^a/K$ has $\Z[a_1,\dotsc,a_n]$-basis given by $\Schub_w(x;a)$ for $w\in S_\infty\setminus S_n$.
The quotient ring $\Z[x,a]/K$ has $\Z[a_1,\dotsc,a_n]$-basis given by all standard monomials $e_I$ for $I=(i_1,i_2,\dotsc)$ with
$0\le i_p \le p$ for all $p \ge1$ and almost all $i_p$ zero, and also by
all double Schubert polynomials $\Schub_w(x;a)$ for $w\in S_\infty$.
But the standard monomials $e_I$ with $i_n=i_{n+1}=\dotsm=0$
are in graded bijection with the $\Schub_w(x;a)$ for $w\in S_n$. It follows that $I^a=J^a$ by graded dimension counting.

Let $J_\infty^{qa}$ be the ideal of $\Z[x,q,a]$
generated by $E_i^p-e_i^p(a)$ for all $i\ge1$ and $p\ge n$, together with $q_i$ for $i\ge n$ 
and $a_i$ for $i > n$. We wish to show that 
\begin{align}\label{E:bigqSchubideal}
  \qSchub_w(x;a)\in J_\infty^{qa}\qquad\text{for all $w\in S_\infty\setminus S_n$}.
\end{align}
For this it suffices to show that 
\begin{align}\label{E:thetaimage}
\theta(J^a_\infty) \subset J^{qa}_\infty.
\end{align}


To prove \eqref{E:thetaimage} it suffices to show that 
\begin{align}\label{E:thetastdgen}
  \theta(e_I(e_i^p(x)-e_i^p(a))) \in J^{qa}_\infty\qquad\text{for standard monomials $e_I$, $i\ge1$ and $p\ge n$.}
\end{align}
To apply $\theta$ to this element we must express $e_I e_i^p$ in standard monomials.
The only nonstandardness that can occur is if $i_p > 0$. In that case one may use \cite[(3.2)]{FGP}:
\begin{align*}
  e_i^p e_j^p &= e_{i-1}^p e_{j+1}^p +  e_j^p e_i^{p+1}- e_{i-1}^p e_{j+1}^{p+1}.
\end{align*}
Note that ultimately the straightening of $e_I e_i^p$ into standard monomials,
only changes factors of the form $e_k^q$ for $k\ge1$ and $q\ge p$.

Let $E_I = \prod_{r\ge1} E_{i_r}^r$ for $I=(i_1,i_2,\dotsc)$. 
If we consider $E_I (E_i^p - e_i^p(a))$ and use \cite[(3.6)]{FGP}
\begin{align*}
    E_i^p E_j^p  &= E_{i-1}^pE_{j+1}^p  +
  E_j^p E_i^{p+1}-E_{i-1}^p E_{j+1}^{p+1} +  q_p( E_{j-1}^{p-1} E_{i-1}^p- E_{i-2}^{p-1} E_j^p).
\end{align*}
to rewrite it into quantized standard monomials, we see that
the two straightening processes differ only by multiples of $q_p$, $q_{p+1}$, etc. 
Therefore 
\begin{align*}
  \theta(e_I (e_i^p(x)-e_i^p(a))) - E_I(E_i^p-e_i^p(a)) \in J_\infty^{qa}
\end{align*} 
But $E_I(E_i^p-e_i^p(a))\in J_\infty^{qa}$ so \eqref{E:thetastdgen} holds
and \eqref{E:bigqSchubideal} follows.

The ring $\Z[x,q,a]/J_\infty^{qa}$ has a $\Z[q_1,\dotsc,q_{n-1};a_1,\dotsc,a_n]$-basis
given by $\qSchub_w(x;a)$ for $w\in S_n$. This follows from Lemmata \ref{L:lead} and \ref{L:code}.
Moreover this basis satisfies the equivariant quantum Chevalley-Monk rule for $SL_n$ by Proposition \ref{P:QEChevalley}.
By Theorem \ref{T:QHTchar} there is an isomorphism of
$\Z[q_1,\dotsc,q_{n-1};a_1,\dotsc,a_n]$-algebras $QH^T(SL_n/B)\to \Z[x,q,a]/J_\infty^{qa}$.
Moreover, $\sigma^w_T$ and $\qSchub_w(x;a)$ (or rather, its preimage in $QH^T(\Fl)$)
are related by an automorphism of $QH^T(\Fl)$. 
But the Schubert divisor class $\sigma^{s_i}_T$ is (by definition) represented by a usual double Schubert polynomial $\Schub_{s_i}(x;a) = \qSchub_{s_i}(x;a)$ (Lemma \ref{L:qdsreflection}) in Kim's presentation, and these divisor classes generate $QH^T(\Fl)$ over 
$\Z[q_1,\dotsc,q_{n-1};a_1,\dotsc,a_n]$.  Thus the automorphism must be the identity, completing the proof.

\section{Parabolic case}
\label{S:para}
\subsection{Notation}
Fix a composition $(n_1,n_2,\dotsc,n_k)\in\Z_{>0}^k$ with $n_1+n_2+\dotsm+n_k=n$.
Let $P\subset SL_n(\C)$ be the parabolic subgroup consisting of block upper triangular
matrices with block sizes $n_1,n_2,\dotsc,n_k$. Then $SL_n/P$ is isomorphic to the
variety of partial flags in $\C^n$
with subspaces of dimensions $N_j:=n_1+n_2+\dotsm+n_j$ for $0\le j\le k$.
Denote by $W_P$ the Weyl group for the Levi factor of $P$ and $W^P$ the
set of minimum length coset representatives in $W/W_P$.
For every $w\in W$ there exists unique elements $w^P\in W^P$ and $w_P\in W_P$ such that
$w=w^P w_P$; moreover this factorization is length-additive. Let $w_0\in W$ be the longest element and 
let $w_0=w_0^Pw_{0,P}$ so that $w_{0,P}\in W_P$ is the longest element.

\begin{example}\label{X:parabolic} Let $(n_1,n_2,n_3)=(2,1,3)$. Then $(N_1,N_2,N_3)=(2,3,6)$,
$w_0^P=564123$ (that is, $w\in S_6 $ is the permutation with $w(1)=5$, $w(2)=6$, etc.),
and $w_{0,P}=213654$.
\end{example}

\subsection{Parabolic quantum double Schubert polynomials}

Let $\Z[x_1,\dotsc,x_n;q_1,\dotsc,q_{k-1}]$ be the graded polynomial ring
with $\deg(x_i)=1$ and $\deg(q_j)=n_j+n_{j+1}$ for $1\le j\le k-1$.

Following \cite{AS} \cite{Cio2} let $D=D^P$ be the $n\times n$ matrix with entries $x_i$ on the diagonal,
$-1$ on the superdiagonal, and entry $(N_{j+1},N_{j-1}+1)$ given by $-(-1)^{n_{j+1}} q_j$
for $1\le j\le k-1$. For $1\le j\le k$ let $D_j$ be the upper left
$N_j \times N_j$ submatrix of $D$ and for $1\le i\le N_j$ define
the elements $G_i^j \in \Z[x;q]$ by 
\begin{align*}
  \det(D_j-t\,\Id) = \sum_{i=0}^{N_j} (-t)^{N_j-i} G_i^j.
\end{align*}
The polynomial $G_i^j$ is homogeneous of degree $i$. 
For $w\in W^P$, we define the parabolic quantum double Schubert polynomial
$\qSchub^P_w(x;a)$ by 
\begin{align}
  \qSchub_w^P(x;a) &= (-1)^{\ell(w (w_0^P)^{-1})}\partial_{w (w_0^P)^{-1}}^a \prod_{j=1}^{k-1} \prod_{i=n-N_{j+1}+1}^{n-N_j} \det(D_j-a_i \Id).
\end{align}

\begin{example} \label{X:parabolicqSchub}
Continuing Example \ref{X:parabolic} we have
\begin{align*}
\qSchub^P_{w_0^P}(x;a) &= \det(D_1-a_4\Id) \det(D_2-a_1\Id)\det(D_2-a_2\Id)\det(D_2-a_3\Id)\\
&=(x_1-a_4)(x_2-a_4)\prod_{i=1}^3((x_1-a_i)(x_2-a_i)(x_3-a_i)+q_1).
\end{align*}
\end{example}

The parabolic quantum double Schubert polynomials $\qSchub^P_w(x;a)$
have specializations similar to the quantum double Schubert polynomials. Let $w\in W^P$.
\begin{enumerate}
\item We define the parabolic quantum Schubert polynomials by the specialization
$\qSchub^P_w(x)=\qSchub^P_w(x;0)$
which sets $a_i=0$ for all $i$. In Lemma \ref{L:qparaspec} it is shown that
these polynomials coincide with those of Ciocan-Fontanine \cite{Cio2},
whose definition uses a parabolic analogue of the quantization map of \cite{FGP}.
\item Setting $q_i=0$ for all $i$ one obtains the double Schubert polynomial
$\Schub_w(x;a)$.
\item Setting both $a_i$ and $q_i$ to zero one obtains the Schubert polynomial
$\Schub_w(x)$.
\end{enumerate}

Let $J_P$ be the ideal in $\Z[x]^{W_P}$ (resp. $J_P^a\subset S[x]^{W_P}$, 
$J_P^q\subset \Z[x]^{W_P}[q]$, $J_P^{qa}\subset S[x]^{W_P}[q]$)
generated by the elements $e_i^n(x)$, (resp. 
$e_i^n(x)-e_i^n(a)$, $G_i^k$, $G_i^k-e_i^n(a)$) for $1\le i\le n$.  The aim of this section is to establish (4) of the following theorem.

\begin{theorem} \label{T:mainpara} \
\begin{enumerate}
\item There is an isomorphism of $\Z$-algebras \cite{BGG, LS}
\begin{align*}
H^*(SL_n/P) &\cong \Z[x]^{W_P}/J_P \\
[X_w] &\mapsto \Schub_w(x)+J_P.
\end{align*} 
\item There is an isomorphism of $S$-algebras \cite{Bi}
\begin{align*}
H^T(SL_n/P) &\cong S[x]^{W_P}/J_P^a \\
[X_w]_T &\mapsto \Schub_w(x;a)+J_P^a.
\end{align*}
\item There is an isomorphism of $\Z[q]$-algebras \cite{AS} \cite{Kim2} \cite{Kim3}
\begin{align*}
QH^*(SL_n/P) &\cong \Z[x]^{W_P}[q]/J_P^q \\
\sigma^{P,w} &\mapsto \qSchub_w^P(x)+J_P^q.
\end{align*}
\item
There is an isomorphism of $S[q]$-algebras
\begin{align}
\label{E:parabquanteqiso}
QH^T(SL_n/P) &\cong S[x]^{W_P}[q]/J_P^{qa} \\
\label{E:parabquanteqisoclass}
\sigma^{P,w}_T&\mapsto \qSchub_w^P(x;a) + J_P^{qa}.
\end{align}
\end{enumerate}
Here $[X_w]$, $[X_w]_T$, $\sigma^{P,w}$, and $\sigma^{P,w}_T$,
denote the Schubert bases for their respective cohomology rings for $w\in W^P$.
\end{theorem}

The isomorphism \eqref{E:parabquanteqiso} is due to \cite{Kim3}. We shall establish \eqref{E:parabquanteqisoclass}, namely, that under this isomorphism,
the parabolic quantum double Schubert polynomials
are the images of parabolic equivariant quantum Schubert classes.

\subsection{Stability of parabolic quantum double Schubert polynomials}
\label{SS:parabstable}
The following Lemma can be verified by direct computation and induction.

\begin{lem} Let $\beta=(\beta_1,\dotsc,\beta_n)\in\Z_{\ge0}^n$ be such that 
$\beta_i \le n-i$ for $2\le i\le n$. We have
\begin{align*}
  \partial_{n-1}^a\dotsm \partial_2^a \partial_1^a \cdot a^\beta &=
  \begin{cases}
  0 &\text{if $\beta_1 <n-1$} \\
  a_1^{\beta_2}a_2^{\beta_3}\dotsm a_{n-1}^{\beta_n} &\text{if $\beta_1=n-1$.}
  \end{cases}
\end{align*}
\end{lem}

Suppose $w\in W^P$ is such that $w(r)=r$ for $N_{k-1}<r\le n$.
Let $w_0^{(p,q)}$ be the minimum length coset representative
of the longest element in $S_{p+q}/(S_p\times S_q)$ and let
$w_0^{P_-}$ be the minimum length coset representative of the longest
element in $S_{N_{k-1}}/(S_{n_1}\times\dotsm\times S_{n_{k-1}})$.
We have the length-additive factorization $w_0^P = w_0^{(N_{k-1},n_k)} w_0^{P_-}$.
Also $\ell(w (w_0^P)^{-1}) = \ell(w (w_0^{P_-})^{-1})+\ell(w_0^{(n_k,N_{k-1})})$.
Using the above Lemma repeatedly we have
\begin{align*}
(-1)^{\ell(ww_0^P)} \qSchub_w^P(x;a) &=
\partial_{w(w_0^P)^{-1}}^a \prod_{j=1}^{k-1} \prod_{i=n-N_{j+1}+1}^{n-N_j} \det(D_j-a_i\Id) \\
&= \partial_{w w_0^{P_-}}^a \partial_{w_0^{(n_k,N_{k-1})}}^a \left(\prod_{j=1}^{k-2} \prod_{i=n-N_{j+1}+1}^{n-N_j} \det(D_j-a_i\Id)\right)
\prod_{i=1}^{n_k} \det(D_{k-1}-a_i\Id) \\
&= (-1)^{\ell(w_0^{(n_k,N_{k-1})})}\partial_{w w_0^{P_-}}^a w_0^{(n_k,N_{k-1})} \cdot \left(\prod_{j=1}^{k-2} \prod_{i=n-N_{j+1}+1}^{n-N_j} \det(D_j-a_i\Id)\right) \\
&=  (-1)^{\ell(w_0^{(n_k,N_{k-1})})}\partial_{w w_0^{P_-}}^a \prod_{j=1}^{k-2} \prod_{i=N_{k-1}-N_{j+1}+1}^{N_{k-1}-N_j} \det(D_j-a_i\Id) \\
&=  (-1)^{\ell(w_0^{(n_k,N_{k-1})})} (-1)^{\ell(w(w_0^{P_-})^{-1})} \qSchub_w^{P_-}(x;a)
\end{align*}
The final outcome is
\begin{align*}
  \qSchub_w^P(x;a) &= \qSchub_w^{P_-}(x;a).
\end{align*}

This means that if we append a block of size $m$ to $\nd$ and 
append $m$ fixed points to $w\in W^P$, the parabolic quantum double Schubert
polynomial remains the same.

\subsection{Stable parabolic quantization}
This section follows \cite{Cio2}.
Consider an infinite sequence of positive integers $\nd=(n_1,n_2,\dotsc)$.
Let $N_j=n_1+\dotsm+n_j$ for $j\ge 1$. Let $W=S_\infty =\bigcup_{n\ge1} S_n$
be the infinite symmetric group (under the inclusion maps $S_n\to S_{n+1}$
that add a fixed point at the end), $W_P$ the subgroup of $W$
generated by $s_i$ for $i\notin \{N_1,N_2,\dotsc\}$, $W^P$ the set of minimum length
coset representatives in $W/W_P$, etc. Let $\Y^P$ be the set of tuples of partitions
$\ladot=(\la^{(1)},\la^{(2)},\dotsc)$ such that $\la^{(j)}$ is contained in the rectangle
with $n_{j+1}$ rows and $N_j$ columns and almost all $\la^{(j)}$ are empty.
Define the standard monomial by
\begin{align*}
  g_\ladot = \prod_{j\ge1} \prod_{i=1}^{n_{j+1}} g_{\la^{(j)}_i}^j
\end{align*}
where $g_r^j=e_r^{N_j}(x)$. The following is a consequence of \cite{Cio2} by taking a limit.

\begin{prop} \cite{Cio2} $\{g_\ladot\mid \ladot\in\Y^P \}$ is
a $\Z$-basis of $\Z[x]^{W_P}$.
\end{prop}

We also observe that

\begin{lem} $\{\Schub_w(x)\mid w\in W^P \}$ is a $\Z$-basis
of $\Z[x]^{W_P}$.
\end{lem}

Define the $\Z[q]$-module automorphism $\qt^P$ of $\Z[x,q]$ by
$\qt(g_\ladot)=G_\ladot$ where 
\begin{align*}
  G_\ladot= \prod_{j\ge1} \prod_{i=1}^{n_{j+1}} G_{\la^{(j)}_i}^j.
\end{align*}
By $\Z[a]$-linearity it defines a $\Z[q,a]$-module automorphism of $\Z[x,q,a]$,
also denoted $\qt^P$.

\begin{lem} \label{L:qparaspec} For all $w\in W^P$,
\begin{align}
\label{E:paraquant}
  \qSchub_w^P(x) &= \qt^P(\Schub_w(x)) \\
\label{E:paraquantequiv}
  \qSchub_w^P(x;a) &= \qt^P(\Schub_w(x;a)).
\end{align}
\end{lem}
\begin{proof} Let $w\in W^P$. We may compute $\qSchub_w^P(x;a)$ by working with respect to
$(n_1,n_2,\dotsc,n_k)$ for some finite $k$; the result is independent of $k$ by the previous subsection.
Then \eqref{E:paraquantequiv} is an immediate consequence of the commutation of $\qt^P$
and the divided difference operators in the $a$ variables. Equation \eqref{E:paraquant}
follows from \eqref{E:paraquantequiv} by setting the $a$ variables to zero.
\end{proof}


\subsection{Cauchy formula}
Keeping the notation of the previous subsection,
we have

\begin{prop} For all $w\in W^P$ we have
\begin{align}
  \qSchub^P_w(x;a) = \sum_{v\wleq w} \Schub_{vw^{-1}}(-a) \qSchub^P_v(x).
\end{align}
\end{prop}
\begin{proof} Observe that if $v\wleq w$ then $v\in W^P$ so that $\qSchub^P_v(x)$ makes sense. We have
\begin{align*}
  \qSchub^P_w(x;a) &= \qt^p(\Schub_w(x;a)) \\
  &= \qt^P(\sum_{v\wleq w} \Schub_{vw^{-1}}(-a)\Schub_v(x)) \\
  &= \sum_{v\wleq w} \Schub_{vw^{-1}}(-a) \qSchub_v^P(x).
\end{align*}
\end{proof}

\subsection{Parabolic Chevalley-Monk rules}
Fix $\nd=(n_1,n_2,\dotsc,n_k)$ with $\sum_{j=1}^k n_j=n$
and let $P$ be the parabolic defined by $\nd$ and so on.
Let $Q^\vee_P$, $\Phi^+_P$, and $\rho_P$ be respectively the coroot lattice, the
positive roots, and the half sum of positive roots, for the Levi factor of $P$.
Let $\eta_P:Q^\vee\to Q^\vee/Q^\vee_P$ be the natural projection.
Let $\pi_P:W\to W^P$ be the map $w\mapsto w^P$ where $w=w^P w_P$ and $w^P\in W^P$
and $w_P\in W_P$. Define the sets of roots
\begin{align}
\label{E:paraA}
  A_{P,w}^n &= \{\alpha\in\Phi^+\setminus \Phi^+_P\mid \text{$w s_\alpha \gtrdot w$ and $ws_\alpha\in W^P$} \} \\
\label{E:paraB}
  B_{P,w}^n &= \{\alpha\in\Phi^+\setminus \Phi^+_P\mid \text{$\ell(\pi_P(w s_\alpha))=\ell(w)+1-\ip{\alpha^\vee}{2(\rho-\rho_P)}$} \}.
\end{align}
For $\la = \sum_{i=1}^{k-1} b_i \alpha_{N_i}^\vee \in Q^\vee/Q_P^\vee$ with $b_i
\in \Z_{\geq 0}$ we let $q_\lambda = \prod_i
q_i^{b_i}$.  The following is Mihalcea's characterization of $QH^T(SL_n/P)$ which extends Theorem \ref{T:QHTchar}.

\begin{theorem}\label{T:Mipara}\cite{Mi} For $w\in W^P$ and $i\in\{N_1,N_2,\dotsc,N_{k-1}\}$ (that is, $s_i\in W^P$), we have
\begin{align*}
  \sigma^{P,s_i}_T \sigma^{P,w}_T &= (-\omega_i(a)+w\cdot \omega_i(a)) \sigma^{P,w}_T \\
  &+ \sum_{\alpha\in A_{P,w}^n} \ip{\alpha^\vee}{\omega_i} \sigma^{P,ws_\alpha}_T +
  \sum_{\alpha\in B_{P,w}^n} \ip{\alpha^\vee}{\omega_i} q_{\eta_P(\alpha^\vee)} \sigma^{P,\pi_P(ws_\alpha)}_T.
\end{align*}
Moreover these structure constants determine the Schubert basis $\{\sigma^{P,w}_T\mid w\in W^P\}$
and the ring $QH^T(SL_n/P)$ up to isomorphism as $\Z[q_1,\dotsc,q_{k-1};a_1,\dotsc,a_n]$-algebras.
\end{theorem}

We now move to the context in which $(n_1,n_2,\dotsc)$ is an infinite sequence of positive integers.
Let $A_{P,w}$ and $B_{P,w}$ be the analogues of the sets of roots defined in \eqref{E:paraA} and \eqref{E:paraB}
where $\rho=(0,-1,-2,\dotsc)$ and
$\rho_P$ is the juxtaposition of $(0,-1,\dotsc,1-n_j)$ for $j\ge1$.

\begin{prop}\label{P:pqeChev} For $w\in W^P$ and $i$ such that $i\in \{N_1,N_2,\dotsc,\}$
(that is, $s_i\in W^P$), the parabolic quantum double Schubert polynomials satisfy the
identity
\begin{align*}
  \qSchub^P_{s_i}(x;a) \qSchub_w^P(x;a) &= (-\omega_i(a)+w\cdot \omega_i(a)) \qSchub_w^P(x;a) +
  \sum_{\alpha\in A_{P,w}} \ip{\alpha^\vee}{\omega_i} \qSchub_{ws_\alpha}^P(x;a) \\
  &+ \,\sum_{\alpha\in B_{P,w}} \ip{\alpha^\vee}{\omega_i} q_{\eta_P(\alpha^\vee)} \qSchub_{\pi_P(ws_\alpha)}^P(x;a).
\end{align*}
\end{prop}
\begin{proof}[Sketch of proof] The proof proceeds entirely analogously to that of 
Proposition \ref{P:QEChevalley}, starting with the Chevalley-Monk formula
for double Schubert polynomials (applied in the special case that $w\in W^P$ and $s_i\in W^P$)
and reducing to the following identity:
\begin{align*}
  & \,\,\sum_{v\wleq w} \Schub_{vw^{-1}}(-a) \sum_{\alpha\in B_{P,v}} q_{\eta_P(\alpha^\vee)} \ip{\alpha^\vee}{\omega_i}
  \qSchub_{\pi_P(vs_\alpha)}^P(x) \\
  &= \sum_{\alpha\in B_{P,w}} \ip{\alpha^\vee}{\omega_i} q_{\eta_P(\alpha^\vee)} \sum_{v\wleq \pi_P(ws_\alpha)}
  \Schub_{v(\pi_P(ws_\alpha))^{-1}}(-a) \qSchub^P_v(x).
\end{align*}
For this it suffices to show that there is a bijection $(v,\alpha)\mapsto(\pi_P(vs_\alpha),\alpha)$
and inverse bijection $(u,\alpha)\mapsto (\pi_P(us_\alpha),\alpha)$ between the sets
\begin{align*}
  A &= \{(v,\alpha)\in W^P \times (\Phi^+\setminus \Phi^+_P)\mid \text{$v\wleq w$ and $\alpha\in B_{P,v}$} \} \\
  B &= \{(u,\alpha)\in W^P \times (\Phi^+\setminus \Phi^+_P)\mid \text{$u\wleq \pi_P(ws_\alpha)$ and $\alpha\in B_{P,w}$} \}
\end{align*}
such that for $(v,\alpha)\in A$ we have
\begin{align}\label{E:Schubindex}
  vw^{-1} = \pi_P(vs_\alpha) (\pi_P(ws_\alpha))^{-1}.
\end{align}
To establish this bijection we use \cite[Lemma 10.14]{LamSh:qaf}, which asserts that elements $\alpha\in B_{P,w}$ 
automatically satisfy the additional condition $\ell(ws_\alpha) = \ell(w) + 1 - \ip{\alpha^\vee}{2\rho}$.

Let $(v,\alpha)\in A$. Then we have length-additive factorizations
$w = (wv^{-1})\cdot v$, $v=(vs_\alpha)\cdot s_\alpha$, and $vs_\alpha=\pi_P(vs_\alpha) \cdot x$ for some $x\in W_P$.
Therefore $w=(wv^{-1})\cdot \pi_P(vs_\alpha) \cdot x \cdot s_\alpha $ is length-additive.
This implies that $ws_\alpha = (wv^{-1})\cdot \pi_P(vs_\alpha) \cdot x$ is length-additive.
One may deduce from this that $\pi_P(ws_\alpha)=(wv^{-1})\pi_P(vs_\alpha)$ and therefore that
\eqref{E:Schubindex} holds. The conditions that $(\pi_P(vs_\alpha),\alpha)\in B$ are readily verified.
This shows that the map $A\to B$ is well-defined. The rest of the proof is similar.
\end{proof}

\subsection{Proof of Theorem \ref{T:mainpara}(4)}
The proof is analogous to the proof of Theorem \ref{T:main}.

Given the finite sequence $(n_1,\dotsc,n_k)$ and parabolic subgroup $P$,
consider the extension $(n_1,\dotsc,n_k,1,1,1,\dotsc)$ by an infinite sequence of $1$s. 
We use the notation $P_\infty$ to label the corresponding objects. That is,
$W_P\cong W_{P_\infty}$ is the subgroup of $S_\infty$, $W^{P_\infty}$ is the set of minimum-length coset representatives
in $S_\infty / W_{P_\infty}$, and so on.

Let $R_P=\Z[x,a]^{W_{P_\infty}}$ and $R^q_P = \Z[x,q,a]^{W_{P_\infty}}$,
where $x = (x_1,x_2,\dotsc)$, $q = (q_1,q_2,\ldots)$, and $q_1,q_2,\dotsc,q_{k-1}$ are identified with the quantum parameters in Theorem \ref{T:mainpara}(4).  Define $J^a_{P,\infty}$ to be the ideal of $R_P$ generated by $g_i^p(x)-e_i^{N_p}(a)$ for $i\ge1$ and $p\ge k$
and by $a_i$ for $i>n$. Note that for $p=k+j$ with $j\ge0$ we have $g_i^p(x)-e_i^{N_p}(a)=e_i^{n+j}(x)-e_i^{n+j}(a)$.

Let $I^a_{P,\infty}$ be the $\Z[a]$-submodule of $R_P$ spanned by $\Schub_w(x;a)$ for $w \in W^{P_\infty} \setminus S_n$, and by $a_i\Schub_u(x;a)$ for $i > n$ and any $u\in W^{P_\infty}$.  We claim that $I^a_{P,\infty}$ is an ideal.  This follows easily from the fact that the only double Schubert polynomials occurring  in the expansion of a product $\Schub_w(x;a) \Schub_v(x;a)$ lie above $w$ and $v$ in Bruhat order.  But then it follows from Theorem \ref{T:mainpara}(2) and dimension counting that $I^a_{P,\infty} = J^a_{P,\infty}$.

Define $J^{qa}_{P,\infty}$ to be the ideal of $R^q_P$ generated by $G_i^p-e_i^{N_p}(a)$ for $p\ge k$ and $i\ge 1$,
and $a_i$ for $i > n$, and $q_i$ for $i \ge k$. We claim that $\qSchub^P_w(x;a) \in J^{qa}_{P,\infty}$ for $w \in W^{P_\infty} \setminus S_n$.  This would follow from the definitions if we could establish that $\theta^{P_\infty}(J^a_{P,\infty}) \subset J^{qa}_{P,\infty}$.  Since $\theta^{P_\infty}$ is trivial on the equivariant variables $a_1,a_2,\ldots$, it suffices to show that $\qt^{P_\infty}(g_{\ladot}^P(g_i^p(x)-e_i^{N_p}(a))) \in J^{qa}_{P,\infty}$, for each $i\ge1$, $p\ge k$, and each standard monomial $g_{\ladot}^P$.  To apply $\qt^{P_\infty}$ to $g_{\ladot}^P g_i^p$ we must first standardize this monomial.  This can be achieved by using a parabolic version of the straightening algorithm of \cite[Proposition 3.3]{FGP}.  The only nonstandard part of $g_{\ladot}^P g_i^p$ is the possible presence of a factor $g_j^p g_i^p$.  This can be standardised using only the (non-parabolic) relation \cite[(3.2)]{FGP} -- any factors $g_m^{k'}$ for $k'<k$ in $g_{\ladot}^P g_i^p$ are not modified.  Similarly, the product $G_{\ladot}^P G_i^p$ can be standardized using the following variant of \cite[Lemma 3.5]{FGP} which can be deduced from \cite[(3.5)]{Cio2} :
$$
G^p_i G^{p+1}_{j+1}+ G^p_{i+1}G^p_{j} \pm q_p G^{p-1}_{i-n_{p+1}-n_p} G^p_j
= G^p_j G^{p+1}_{i+1} +G^p_{k+1} G^p_i \pm q_p G^{p-1}_{j-n_{p+1}-n_p}G^p_i.
$$
We note that when $p\ge k$, this relation only involves quantum variables $q_k, q_{k+1}, \dotsc$.
Thus modulo $q_k,q_{k+1},\dotsc$, the straightening relation for $g_{\ladot}^P g_i^p$ and for $G_{\ladot}^PG_i^p$ coincide.  It follows that $\qt^{P_\infty}(g_{\ladot}^P(g_i^p(x)-e_i^{N_p}(a)) = G_{\ladot}^P(G_i^p - e_i^{N_p}(a))  \mod J^{qa}_{P,\infty}$, and thus $\qSchub^P_w(x;a) \in J^{qa}_{P,\infty}$ for $w \in W^{P_\infty} \setminus S_n$.  But $R^{qa}_{P,\infty}/J^{qa}_{P,\infty}$ has rank $|W^P|$ over $Z[q_1,\ldots,q_{k-1},a_1,\ldots,a_n]$ and so it follows that $J^{qa}_{P,\infty}$ is spanned by $\qSchub^P_w(x;a) \in J^{qa}_{P,\infty}$ for $w \in W^{P_\infty} \setminus S_n$ together with $q_i\qSchub^P_u(x;a)$ for $i\ge k$, $a_i\qSchub^P_u(x;a)$ for $i > n$ and any $u\in W^{P_\infty}$. 

Theorem \ref{T:mainpara}(4) now follows from Proposition \ref{P:pqeChev} and the determination of $QH^T(SL_n/P)$ in Theorem \ref{T:Mipara}.

\end{document}